\documentclass{article}
\usepackage[T1]{fontenc}
\RequirePackage[OT1]{fontenc}
\usepackage{amsmath,amsfonts,amssymb,amsthm, color,wrapfig,float,dsfont}
\usepackage{authblk}
\usepackage{url}   
\usepackage{enumerate}
\usepackage{color}
\usepackage{a4wide}
\usepackage[francais,english]{babel}
\title{The Vlasov-Fokker-Planck equation in non-convex landscapes: convergence to equilibrium}
\author[1]{M. H. Duong}
\author[2]{J. Tugaut}
\affil[1]{Mathematics institute, University of Warwick, Coventry CV4 7AL, UK}
\affil[2]{Univ Lyon, UJM-Saint-\'Etienne, CNRS UMR 5208, Institut Camille Jordan, 10 rue Tr\'efilerie, CS 82301, F-42023 Saint-Etienne Cedex 2, France}
\date{\today}

\def\div{\mathop{\mathrm{div}}\nolimits}
\def\!{\mathop{\mathrm{!}}}

\def \R{\mathbb{R}}

\begin{document}

\newcommand{\bRb}{\mathbb{R}}
\newcommand{\bCb}{\mathbb{C}}
\newcommand{\bEb}{\mathbb{E}}
\newcommand{\bKb}{\mathbb{K}}
\newcommand{\bQb}{\mathbb{Q}}
\newcommand{\bFb}{\mathbb{F}}
\newcommand{\bGb}{\mathbb{G}}
\newcommand{\bNb}{\mathbb{N}}
\newcommand{\bZb}{\mathbb{Z}}

\newcommand{\deriv}{\stackrel{\mbox{\bf\Large{}$\cdot$\normalsize{}}}}
\newcommand{\dederiv}{\stackrel{\mbox{\bf\Large{}$\cdot\cdot$\normalsize{}}}}

\theoremstyle{plain} \newtheorem{thm}{Theorem}[section]
\theoremstyle{plain} \newtheorem{prop}[thm]{Proposition}
\theoremstyle{plain} \newtheorem{props}[thm]{Properties}
\theoremstyle{plain} \newtheorem{ex}[thm]{Example}
\theoremstyle{plain} \newtheorem{contrex}[thm]{Coounterexample}
\theoremstyle{plain} \newtheorem{cor}[thm]{Corollary}
\theoremstyle{plain} \newtheorem{hyp}[thm]{Hypothesis}
\theoremstyle{plain} \newtheorem{hyps}[thm]{Hypotheses}
\theoremstyle{plain} \newtheorem{lem}[thm]{Lemma}
\theoremstyle{plain} \newtheorem{rem}[thm]{Remark}
\theoremstyle{plain} \newtheorem{nota}[thm]{Notation}
\theoremstyle{plain} \newtheorem{defn}[thm]{Definition}

\newcommand{\cRc}{\mathcal{R}}
\newcommand{\cCc}{\mathcal{C}}
\newcommand{\cEc}{\mathcal{E}}
\newcommand{\cKc}{\mathcal{K}}
\newcommand{\cQc}{\mathcal{Q}}
\newcommand{\cFc}{\mathcal{F}}
\newcommand{\cGc}{\mathcal{G}}
\newcommand{\cNc}{\mathcal{N}}
\newcommand{\cZc}{\mathcal{Z}}
\newcommand{\cOc}{\mathcal{O}}
\newcommand{\cSc}{\mathcal{S}}
\newcommand{\cAc}{\mathcal{A}}
\newcommand{\cBc}{\mathcal{B}}
\newcommand{\cIc}{\mathcal{I}}
\newcommand{\cDc}{\mathcal{D}}
\newcommand{\cLc}{\mathcal{L}}
\newcommand{\cHc}{\mathcal{H}}

\newcommand{\Ima}{{\rm Im}}
\newcommand{\Rea}{{\rm Re}}
\newcommand{\gaga}{\left|\left|}
\newcommand{\drdr}{\right|\right|}
\newcommand{\lra}{\left\langle}
\newcommand{\rra}{\right\rangle}

\newcommand{\bal}{\begin{align}}
\newcommand{\eal}{\end{align}}
\newcommand{\beq}{\begin{equation}}
\newcommand{\eeq}{\end{equation}}

\newcommand{\ba}{\begin{align*}}
\newcommand{\be}{\begin{equation*}}
\newcommand{\ee}{\end{equation*}}

\newcommand{\EE}{\mathbb{E}}
\newcommand{\PP}{\mathbb{P}}

\newcommand{\sepa}{\left|\right.}

\newcommand{\crg}{[\![}
\newcommand{\crd}{]\!]}

\newcommand{\sgn}{{\rm Sign}}
\newcommand{\vari}{{\rm Var}}
\newcommand{\cov}{{\rm Cov}}
\newcommand{\poin}[1]{\dot{#1}}
\newcommand{\norm}[1]{\Vert #1\Vert}

\maketitle
\begin{abstract}
In this paper, we study the long-time behaviour of solutions to the Vlasov-Fokker-Planck equation where the confining potential is non-convex. This is a nonlocal nonlinear partial differential equation describing the time evolution of the probability distribution of a particle moving under the influence of a non-convex potential, an interaction potential, a friction force and a stochastic force. Using the free-energy approach, we show that under suitable assumptions solutions of the Vlasov-Fokker-Planck equation converge to an invariant probability.  
\end{abstract}
\medskip

 \par\medskip

\par\medskip

\section{Introduction}
\label{sec:intro}

\subsection{The Vlasov-Fokker-Planck equation}
This paper concerns with the convergence to equilibrium of the following (dimensionless) \emph{Vlasov-Fokker-Planck (VFP)} equation
\begin{equation}
\label{eq:VFP}
\partial_t \rho = - \div_q\Big(\rho p\Big) + \div_p\left[\rho\left(\nabla_q V + \nabla_q F\ast \rho + p\right)\right]+\lambda\Delta_p \rho.
\end{equation}
for the probability density function $\rho(t,q,p)$ in $ \R_+\times\R^d\times\R^d$. In the above equation, subscripts as in $\div_q$ and $\Delta_p$ indicate that the differential operators act only on those variables. The functions $V = V(q)$ and $F = F(q)$ are given. The convolution $F\ast \rho$ is defined by $(F\ast\rho)(q)=\int_{\mathbb{R}^{2d}}F(q-q')\rho(q',p')\,dq'dp'$. Finally $\lambda$ is a positive constant.

Equation~\eqref{eq:VFP} describes the evolution of the probability density $\rho(t,q,p)$ for the following self-stabilizing diffusion process in the phase space 
\begin{align}
\label{eq: SDE}
&dQ(t)= P(t)\, dt,\nonumber
\\&dP(t)=-\nabla V(Q(t))\,dt-\nabla F\ast \rho_t(Q(t))\,dt-P(t)\,dt+\sqrt{2\lambda}\, dW(t).
\end{align}
This is a nonlinear diffusion since the own law of the process intervenes in the drift. It models the movement of a particle under a fixed potential $V$, an interaction potential $F$, a friction force (the drift term $-P(t)\,dt$) and a stochastic forcing described by the $d$-dimensional Wiener measures~$W(t)$. Equation \eqref{eq:VFP} and system \eqref{eq: SDE} has been widely used in chemistry and statistical mechanics such as a model for
chemical reactions or as a model for particles interacting through Coulomb or gravitational forces \cite{Kramers40, BD95}.

\subsection{Literature overview and the aim of the present paper}
Closely related to the VFP equation is the McKean-Vlasov equation
\[
\partial_t \hat\rho(t,x)=\div[\hat\rho(t,x)(\nabla V(x)+(F\ast\hat\rho_t)(x))]+\lambda \Delta\hat\rho(t,x),
\]
which is the corresponding spatial homogeneous model and can also be obtained from the former in the overdamped limit \cite{DLPS2017}. The trend to equilibrium for the McKean-Vlasov equation has been studied extensively in the literature due to the multiplicity of assumptions that one can impose on the confining potential $V$ and the interaction potential $F$. When $V$ and $F$ are convex, the McKean-Vlasov has a unique stationary solution and one can obtain explicit rate of convergence to the equilibrium, see e.g., \cite{CMV2006, BGG13} and references therein. When the potential $V$ is a double-well and the interaction potential $F$ is quadratic, Dawson \cite{Dawson83} proved, among other things, that the McKean-Vlasov equation exhibits a phase transition phenomenon, that is it may have a unique stationary solution or several ones when the diffusion coefficient is above or below a critical value. Around the same time, Tamura proved \cite{Tamura1984} the existence of a phase transition and investigated the order of the convergence to an equilibrium in the case when $V$ is a quadratic potential perturbed by a rapidly decreasing function and $F$ is also a rapidly decreasing function. Later on, Tugaut and Herrmann proved the thirdness of stationary solutions of the McKean-Vlasov equation when $V$ is a double-well potential and $F$ is an even polynomial in one-dimensional space. Subsequently, in the same setting, Tugaut proved the convergence to an equilibrium \cite{AOP} and the existence of a phase transition \cite{PT}. He then extended the non-uniqueness and convergence of stationary solutions to general multi-wells lansdcape in general dimension \cite{SPA,JTP}. More recently, Barbaro and co-authors \cite{BCCD16} have demonstrated the existence of a phase transition for the case of (a family of) double-well potential and quadratic interaction in an arbitrary dimension by verifying that the stability of the isotropic equilibria changes as the diffusion coefficient 
crosses a threshold value.

The study of the long-time behaviour for the VFP equation is often more difficult than that of the McKean-Vlasov equation because of two reasons: (1) it is a degenerate diffusion process where the Laplacian acts only on the $p$ variable and (2) it is not a gradient flows but simultaneously presents both Hamiltonian and gradient flows effects. When $V$ and $F$ are both convex, similarly as in the McKean-Vlasov case, the VFP has a unique invariant measure and one can also get explicit rate of convergence to the equilibrium, see \cite{BGM10,Mon2017} (see also \cite{Duong15NA} for a similar system). The case where $V$ and $F$ are non-convex is much less known. In \cite{DuongTugaut2016}, we showed that, non-uniqueness (in multi-well landscapes and general dimensional setting) and phase transition phenomena (in one dimensional case) of invariant measures also occur in the full VFP equation.

In this paper, we continue considering the VFP equation in the non-convex setting. In the main theorem, Theorem \ref{theo: main result} below, we show that under suitable assumptions, solutions of the VFP equation converge to an invariant probability. We employ the free-energy approach combining techniques that have been used in \cite{BCS97} for the Vlasov-Poisson-Fokker-Planck system and in \cite{AOP,SPA} for the McKean-Vlasov system. Several  technical improvements will be carried out to overcome the two difficulties mentioned above.
\subsection{Organization of the paper}
The rest of the paper is organized as follows. We state assumptions and the main result in Section~\ref{sec: asspt and main result}. The proof of the main result is given in Section~\ref{sec: proof of main thm}.
\section{Assumptions and statement of the main result}
\label{sec: asspt and main result}
We now give the exact assumptions of the paper. We take similar hypotheses as the ones of \cite{AOP,SPA}. Throughout this paper, $\gaga.\drdr$ denotes the Euclidean norm on $\bRb^d$.\\[4pt]
{\bf Assumption (M). }\emph{We say that the confining potential $V$, the interacting potential $F$ and the initial law $\rho_0$ satisfy the set of assumptions \textbf{(M)} if}\\[2pt]
{\bf (M-1)} \emph{$V$ is a smooth function on $\bRb^d$.}\\
{\bf (M-2)} \emph{There exists a compact subset $\mathcal{K}$ of $\bRb^d$ such that $\nabla^2V(q)>0$, for all $q\notin\mathcal{K}$. There exists a positive constant $m$ such that $||\nabla V(x)||\leq K||x||^{2m}$. Moreover, $\displaystyle\lim_{||q||\to+\infty}\nabla^2V(q)=+\infty$.}\\
{\bf (M-3)} \emph{There exist positive constants $C_2,C_4$ such that $V(q)\geq C_4\gaga q\drdr^4-C_2\gaga q\drdr^2$ for any $q\in\bRb^d$.}\\
{\bf (M-4)} \emph{There exists an even and positive polynomial function $G$ on $\bRb$ such that $F(q)=G(||q||)$. And, $\deg(G)=:2n\geq2$.}\\
{\bf (M-5)} \emph{The function $G$ is convex.}\\
{\bf (M-6)} \emph{The $8r^2$-th moment of the measure $\rho_0$ with respect to the variable $q$ is finite with $r:=\max\left\{m,n\right\}$. And, the second moment of the measure $\rho_0$ with respect to the variable $p$ is finite.}\\
{\bf (M-7)} \emph{The measure $\rho_0$ admits a $\mathcal{C}^\infty$-continuous density $\rho_0$ with respect to the Lebesgue measure. And, the entropy $S(\rho_0)=-\iint_{\bRb^d\times\bRb^d}\rho_0(q,p)\log(\rho_0(q,p))dqdp$ is finite.}\\[6pt]
Let us first give an overview about these assumptions.

Assumption \textbf{(M-1)} is for convenience and one can prove that it is sufficient that $V$ is of class $\mathcal{C}^2$. 

Assumption \textbf{(M-2)} allows the potential $V$ to confine the process and to avoid that it explodes. So it is used for having a solution to the stochastic differential equation.

Assumption \textbf{(M-3)} is used to show that the free-energy functional, see \eqref{eq: free energy}, is bounded from below. 

Assumptions \textbf{(M-4)} and \textbf{(M-5)} are assumed in order to rely on the results in \cite{AOP,SPA} about the description of the set of invariant probabilities for the McKean-Vlasov equation. Also, we need a control of the gradient of $F$ at infinity and a polynomial function is the simplest case. However, the convexity of the function $G$ is not necessary. To have $G''\geq C$ with $C\in\bRb$ would be sufficient.

Assumption \textbf{(M-6)} is necessary in order to ensure that there is a solution to the stochastic differential equation.

Assumption (M-7) ensures that the initial free-energy is finite, which is necessary to show that it converges to something finite.

\begin{defn}
\label{def:set}
By $\mathcal{A}_\sigma$ (resp. $\mathcal{S}_\sigma$), we denote the set of the limiting values of the family $\left\{\rho_t\,;\,t\geq0\right\}$ (resp. the set of the invariant probabilities of the VFP).
\end{defn}

\begin{defn}
\label{def:discrete}
We say that a set $\mathcal{D}$ of measures on $\bRb^d\times\bRb^d$ is discrete if for any $\nu\in\mathcal{D}$, there exists a neighbourhood $\mathcal{V}$ of $\nu$ for the topology of the weak convergence such that $\mathcal{D}\bigcap\mathcal{V}=\left\{\nu\right\}$. In a similar way, we say that $\mathcal{D}$ is path-connected if it is path-connected for the topology of the weak convergence.
\end{defn}

We now are ready to state the main result of the present paper whose proof is given in Section~\ref{sec: proof of main thm}.
\begin{thm}
\label{theo: main result}
Assume that the confining potential $V$, the interacting potential $F$ and the initial law $\rho_0$ satisfy the set of assumptions (M). Then the set $\mathcal{A}_\sigma$ is either a single element $\rho^\sigma\in\mathcal{S}_\sigma$ or a path-connected subset of $\mathcal{S}_\sigma$.
\end{thm}
From this theorem we deduce the following corollary.
\begin{cor}
\label{negan}
Assume that the confining potential $V$, the interacting potential $F$ and the initial law $\rho_0$ satisfy the set of assumptions (M). Assume further that the set $\mathcal{S}_\sigma$ is discrete. Then the probability measure $\rho_t$ converges weakly to an invariant probability $\rho^\sigma\in\mathcal{S}_\sigma$, as $t$ goes to infinity.
\end{cor}
For examples of potentials for which the set of invariant probabilities is discrete, see \cite{DuongTugaut2016}.

This corollary means that if we have a set of invariant probabilities which is finite, thus under simple assumptions, the solution of the VFP equation does converge as time goes to infinity to one of the steady state. In particular, if there is a unique invariant measure with total mass equal to one, there is convergence towards this measure. For example, with the potentials $V(x):=\frac{x^4}{4}-\frac{x^2}{2}$ and $F(x):=\alpha\frac{x^2}{2}$ with $\alpha>0$, it is known that there is either one or three invariant probabilities  \cite{DuongTugaut2016}. From Corollary \ref{negan}, we know that there is a convergence towards one of the invariant probability. There still is a remaining question: in the case in which there are three steady states, which one does the solution of the Vlasov-Fokker-Planck equation converge to? We leave this question for future research.

\section{Proof of the main theorem}
\label{sec: proof of main thm}
In this section, we prove Theorem \ref{theo: main result}. We employ the free-energy approach combining techniques that have been used in \cite{BCS97} for the Vlasov-Poisson-Fokker-Planck system and in \cite{AOP,SPA} for the McKean-Vlasov system. The advantage of this method is that it facilitates the Hamiltonian-gradient flows structure of the VFP equation. Our proof will consist of three steps as follows.
\begin{enumerate}[Step 1)]
\item  We first consider a free-energy functional showing that it is non-increasing along solutions $\rho_t$ of the VFP equation (Lemma \ref{nadege}) and is bounded from below (Lemma \ref{oumaima}).
\item Then we show that one can extract a converging subsequence as $t\to\infty$ from the trajectories $\{\rho_t\}_{t\geq 0}$ (Corollary \ref{oumaima2}, Lemma \ref{hanae2} and Proposition \ref{violette}).
\item In the last step, we characterise the set of stationary probabilities (Lemma \ref{tyler} and Proposition \ref{manue}).
\end{enumerate}
We now follow the strategy. Denote $H(q,p):=\frac{p^2}{2}+V(q)$. For $\rho(dqdp)=\rho(q,p)\,dqdp$, we define the free energy by
\begin{equation}
\label{eq: free energy}
\Upsilon_\lambda(\rho):=\iint_{\mathbb{R}^{2d}}\Big(H(q,p)+\frac{1}{2}F\ast \rho+ \lambda\log \rho\Big)\rho\,dqdp\,.
\end{equation}
The free-energy functional $\Upsilon_\lambda$ plays the role of a Lyapunov function. Indeed, the next lemma shows that the free-energy is non-increasing along the trajectories of the flow $(\rho_t)_{t\geq0}$.
\begin{lem}
\label{nadege}
Let $\rho_t(dqdp)=\rho_t(q,p)\,dqdp$ be a solution of \eqref{eq:VFP}. We put $\eta(t):=\Upsilon_\lambda(\rho_t)$. Then it holds that
\begin{align}
\label{time derivative of the free energy}
\frac{d}{dt}\eta(t)=&-\iint_{\R^{2d}}\frac{1}{\rho_t}\gaga p\rho_t+\lambda\nabla_p\rho_t\drdr^2\,dqdp\\
\label{pareil}
=&-\iint_{\bRb^{2d}}\gaga p\sqrt{\rho_t}+2\lambda\nabla_p\sqrt{\rho_t}\drdr^2\,dqdp\leq0\,.
\end{align}
\end{lem}
\begin{proof}
Note that Equation \eqref{eq:VFP} can be written as
\begin{equation*}
\partial_t\rho_t=\div (J\nabla \mathcal{H}\rho_t)+\div_p\left(p\rho_t+\lambda\nabla_p\rho_t\right),
\end{equation*}
where $\div$ and $\nabla$ denote the divergence and gradient operators respectively with respect to the full spatial coordinate $(q,p)$, $\mathcal{H}(q,p):=H(q,p)+(F\ast \rho_t)(q)$ and
$J=\begin{pmatrix}
0&I\\
-I&0
\end{pmatrix}
$
being the  canonical $2d\times 2d$-symplectic matrix.
This formulation is useful in the following calculation
\begin{align*}
\frac{d}{dt}\eta(t)&=\iint_{\R^{2d}}\left(H(q,p)+F\ast\rho_t(q)+\lambda(\log\rho_t+1)\right)\partial_t\rho_t(q,p)\,dqdp\\
&=\iint_{\R^{2d}}\Big(H(q,p)+F\ast\rho_t(q)+\lambda\log\rho_t\Big)\partial_t\rho_t(q,p)\,dqdp\\
&=\iint_{\R^{2d}}\Big(\mathcal{H}(q,p)+\lambda\log\rho_t\Big)\left(\div (J\nabla \mathcal{H}\rho_t)+\div_p\left(p\rho_t+\lambda\nabla_p\rho_t\right)\right)\,dqdp\\
&\overset{(*)}{=}-\iint_{\R^{2d}}\left(J\nabla\mathcal{H}\cdot\nabla \mathcal{H}\rho_t+\lambda J\nabla\mathcal{H}\cdot\begin{pmatrix}
\nabla_q\rho_t\\
\nabla_p\rho_t
\end{pmatrix}\right)\,dqdp\\
&\quad\quad-\iint_{\R^{2d}}\left(\left(p\rho_t+\lambda\nabla_p\rho_t\right)\cdot \nabla_p\left(\mathcal{H}(q,p)+\lambda\log\rho_t\right)\right)\,dqdp
\\&=-\iint_{\R^{2d}}\frac{1}{\rho_t}\gaga p\rho_t+\lambda\nabla_p\rho_t\drdr^2\,dqdp,
\end{align*}
which proves \eqref{time derivative of the free energy}. Note that the first two terms in $(*)$ have vanished because of the anti-symmetry property of $J$ and integration by parts. Equation \eqref{pareil} is just another representation of \eqref{time derivative of the free energy}.
\end{proof}

Let us point out that the derivative of $\eta$ does not directly provide the form that must satisfy an adherence value of the set $\{\rho_t\,;\,t\geq0\}$. Consequently, obtaining the convergence for VFP equation is more difficult than it has been for the McKean-Vlasov equation (in which $\eta'(t)=0$ if and only if $\rho_t$ is an invariant probability), see \cite{BCCP98,CMV03} for the convex case and \cite{AOP,SPA} for the nonconvex one.

\begin{lem}
\label{oumaima}
The functional $\eta$ is bounded from below by a constant $\Xi_\lambda$.
\end{lem}
\begin{proof}
First, we have 
\begin{align*}
 \iint_{\bRb^{2d}}\left(F\ast\rho_t\right)(q)\rho_t(q,p)dqdp=\iiiint_{\bRb^{2d}\times\bRb^{2d}}F\left(q-q'\right)\rho_t(q,p)\rho_t(q',p')dqdpdq'dp'\geq0\,.
\end{align*}
We deduce that $\Upsilon_\lambda(\rho)\geq\Upsilon_\lambda^-(\rho)$ with
\begin{align*}
\Upsilon_\lambda^-(\rho):=&\lambda\iint_{\bRb^{2d}}\rho(q,p)\log\left(\rho(q,p)\right)\mathds{1}_{\{\rho(q,p)<1\}}dqdp+\iint_{\mathbb{R}^{2d}}\left(\frac{p^2}{2}+V(q)\right)\rho(q,p)dqdp\,.
\end{align*}
It suffices then to prove the inequality $\Upsilon_\lambda^-(\rho)\geq\Xi_\lambda$. We proceed as in the first part of the proof of Theorem 2.1 in \cite{BCCP98}. We show that we can minorate the negative part of the entropy by a function of the second moment (for the variable $q$).\\[2pt]
We split the negative part of the entropy into two integrals: \begin{align*}
&-\iint_{\mathbb{R}^{2d}}\rho(q,p)\log\left(\rho(q,p)\right)\mathds{1}_{\{\rho(q,p)<1\}}dqdp=-I_+-I_-\\
\mbox{with}\quad &I_+:=\iint_{\mathbb{R}^{2d}}\rho(q,p)\log\left(\rho(q,p)\right)\mathds{1}_{\{e^{-\gaga (q,p)\drdr}<\rho(q,p)<1\}}dqdp,\\
\mbox{and}\quad &I_-:=\iint_{\mathbb{R}^{2d}}\rho(q,p)\log\left(\rho(q,p)\right)\mathds{1}_{\{\rho(q,p)\leq e^{-\gaga (q,p)\drdr}\}}dqdp\,.
\end{align*}
By definition of $I_+$, we have the following estimate: 
\begin{align*}
I_+&\geq\iint_{\mathbb{R}^{2d}}\rho(q,p)\log\left(e^{-\gaga (q,p)\drdr}\right)\mathds{1}_{\{e^{-\gaga (q,p)\drdr}<\rho(q,p)<1\}}dqdp\\
&\geq-\iint_{\mathbb{R}^{2d}}\gaga (q,p)\drdr\rho(q,p)\mathds{1}_{\{e^{-\gaga q\drdr}<\rho(q,p)<1\}}dqdp\\
&\geq-\iint_{\mathbb{R}^{2d}}\gaga (q,p)\drdr\rho(q,p)dqdp\geq-\frac{\lambda}{2}-\frac{1}{2\lambda}\iint_{\mathbb{R}^{2d}}(\gaga q\drdr^2+\gaga p\drdr^2)\rho(q,p)dqdp\,.
\end{align*}
By putting $\gamma(x):=\sqrt{x}\log(x)\mathds{1}_{\{x<1\}}$, a simple computation provides $\gamma(x)\geq-2e^{-1}$ for all $x<1$. We deduce:
\begin{eqnarray*}
I_-=\iint_{\mathbb{R}^{2d}}\sqrt{\rho(q,p)}\gamma(\rho(q,p))\mathds{1}_{\{\rho(q,p)\leq e^{-\gaga (q,p)\drdr}\}}dqdp\geq-2e^{-1}\iint_{\mathbb{R}^{2d}}e^{-\frac{\gaga (q,p)\drdr}{2}}dqdp=C(d)\,,
\end{eqnarray*}
where $C(d)$ being a constant which depends only on the dimension $d$.

Consequently, it yields:
\begin{eqnarray*}
-\iint_{\mathbb{R}^{2d}}\rho(q,p)\log\left(\rho(q,p)\right)\mathds{1}_{\{\rho(q,p)<1\}}dqdp\leq\frac{1}{2\lambda}\iint_{\mathbb{R}^{2d}}(\gaga q\drdr^2+\gaga p\drdr^2)\rho(q,p)dqdp+\frac{\lambda}{2}+C(d)\,.
\end{eqnarray*}
This implies:
\begin{eqnarray}
\label{eq:fr:minoration}
\Upsilon_\lambda^-(\rho)\geq C'(\lambda,d)+\iint_{\mathbb{R}^{2d}}\left(V(q)-\frac{\|q\|^2}{2}\right)\rho(q,p)dqdp\,,
\end{eqnarray}
with $C'(\lambda, d)$ being a constant which depends on $\lambda$ and $d$. By hypothesis, there exist $C_2, C_4>0$ such that $V(q)\geq C_4\gaga q\drdr^4-C_2\gaga q\drdr^2$ so the function $q\mapsto V(q)-\frac{\gaga q\drdr^2}{2}$ is lower-bounded by a constant. This achieves the proof.

\end{proof}

We readily obtain the following corollaries:

\begin{cor}
\label{hanae}
There exists a constant $L_\lambda$ which does depend on $\rho_0$ such that $\eta(t)$ converges towards $L_\lambda$ as $t$ goes to infinity.
\end{cor}

\begin{cor}
\label{oumaima2}
For any $T>0$, we have
\begin{equation*}
\lim_{t\to\infty}\int_0^T\eta'(t+s)ds=\lim_{t\to\infty} (\eta(t+T)-\eta(t))=0\,.
\end{equation*}
Therefore, from this together with \eqref{pareil}, we deduce that 
\begin{equation}
\label{meryem}
\lim_{t\to+\infty}\left|\left|\nabla_p\sqrt{\rho^{(t)}}+\frac{p}{2 \lambda}\sqrt{\rho^{(t)}}\right|\right|_{{\rm L}^2\left([0;T]\times\bRb^d\times\bRb^d\right)}=0\,,
\end{equation}
where we put $\rho^{(t)}(s,q,p):=\rho_{t+s}(q,p)$.
\end{cor}

From now on, we consider a sequence $(t_n)_n$ of positive real numbers which converges to infinity. And, by $\rho^n$, we denote $\rho^{(t_n)}$. By $W_n$, we denote the potential $V+F\ast\rho^n$.

We now prove that the family $\rho^n$ is relatively compact in $\mathcal{C}\left([0;T],{\rm L}^1\left(\bRb^d\times\bRb^d\right)\right)$. To do so, we need to show the uniform (with respect to the time) boundedness of the moments of $\rho_t$.

\begin{lem}
\label{hanae2}
There exists $C_0>0$ such that
\begin{equation*}
\sup_{t\geq0}\Upsilon_\lambda(\rho_t)=\sup_{t\geq0}\iint_{\bRb^{2d}}\left(\frac{p^2}{2}+V(q)+\frac{1}{2}F\ast\rho_t(q)\right)\rho_t(q,p)dqdp\leq C_0\,,
\end{equation*}
and
\begin{equation*}
\sup_{t\geq0}\iint_{\bRb^{2d}}\rho_t(q,p)\log\left[\rho_t(q,p)\right]\mathds{1}_{\rho_t(q,p)\geq1}dqdp\leq C_0\,.
\end{equation*}
\end{lem}
\begin{proof}
According to the hypothesis of the paper, we have $\Upsilon_\lambda(\rho_0)<\infty$. And, according to Lemma \ref{nadege}, $\Upsilon_\lambda(\rho_t)\leq\Upsilon_\lambda(\rho_0)$ for any $t\geq0$. We remind the reader Inequality \eqref{eq:fr:minoration}:

\begin{equation*}
\Upsilon_\lambda(\rho)\geq C'(\lambda, d)+\iint_{\mathbb{R}^{2d}}\left(V(q)-\frac{\|q\|^2}{2}\right)\rho(q,p)dqdp\,.
\end{equation*}
We immediately obtain the finiteness of $\displaystyle\sup_{t\geq0}\iint_{\bRb^{2d}}V(q)\rho_t(q,p)dqdp$. In particular, it follows from the assumption (\textbf{M-3}) that the second moment for the variable $q$ of $\rho_t$ is finite. Consequently, by proceeding like in Lemma \ref{oumaima}, we obtain the uniform (with respect to the time) boundedness of $-\iint_{\mathbb{R}^{2d}}\rho_t(q,p)\log\left(\rho_t(q,p)\right)\mathds{1}_{\{\rho_t(q,p)<1\}}dqdp$. We deduce that $\iint_{\mathbb{R}^{2d}}F\ast\rho_t(q)\rho_t(q,p)dqdp$ and $\iint_{\bRb^{2d}}\rho_t(q,p)\log\left[\rho_t(q,p)\right]\mathds{1}_{\rho_t(q,p)\geq1}dqdp$ are bounded, uniformly with respect to the time.
\end{proof}

We now follow the plan of \cite{BCS97}.

\begin{prop}
\label{violette}
The sequence of distributions $\left\{\rho^n\,,\,n\in\mathbb{N}^*\right\}$ is relatively compact in \\ $\mathcal{C}\left([0;T],{\rm L}^1\left(\bRb^d\times\bRb^d\right)\right)$.
\end{prop}
The proof of this Proposition is similar as that of \cite[Lemma 5.2]{BCS97}; hence we omit it here.

Consequently, there exists a subsequence $(t_{\varphi(n)})_n$ of $(t_n)_n$ such that $\rho^{\varphi(n)}$ converges to a function $\rho^\infty$. For the comfort of the reading, we will denote this subsequence as $(t_n)_n$. As a consequence, 
\begin{equation*}
\sqrt{\rho^n}\longrightarrow\sqrt{\rho^\infty}\,,
\end{equation*}
in ${\rm L}^2\left([0;T]\times\bRb^d\times\bRb^d\right)$. Using \eqref{meryem}, we obtain that
\begin{equation}
\label{meryem2}
\nabla_p\sqrt{\rho^{\infty}}+\frac{p}{2\lambda}\sqrt{\rho^{\infty}}=0\,,
\end{equation}
in the sense of distributions on $[0;T]\times\bRb^d\times\bRb^d$. Multiplying \eqref{meryem2} by $\exp\left\{\frac{p^2}{4\lambda}\right\}$, we obtain
\begin{equation*}
\nabla_p\left(\exp\left\{\frac{p^2}{4\lambda}\right\}\sqrt{\rho^{\infty}}\right)=0\,,
\end{equation*}
in the sense of distributions on $[0;T]\times\bRb^d\times\bRb^d$, which implies the existence of a function $g_\infty\in{\rm L}^1_{{\rm loc}}\left([0;T]\times\bRb^d\right)$ such that
\begin{equation*}
\rho^{\infty}(t,q,p)=g_\infty(t,q)Z_\lambda^{-1}\exp\left\{-\frac{p^2}{2\lambda}\right\}\,,
\end{equation*}
where $Z_\lambda$ is the normalising constant such that $\int_{\R^d}e^{-\frac{p^2}{2\lambda}}\,dp=1$.

On the other hand, by proceeding like in \cite{AOP,SPA}, we can prove that $\nabla F\ast\rho^n(q)$ converges uniformly on each compact set towards $\nabla F\ast\rho^\infty(q)$.

\begin{lem}
\label{tyler}
We have:

\begin{equation*}
\frac{\partial}{\partial t}\rho^\infty=-\div_q\Big(\rho^\infty p\Big)+\div_p\Big(\rho^\infty(\nabla_qV+\nabla_qF\ast\rho^\infty+p)\Big)+\lambda\Delta_p\rho^\infty\,.
\end{equation*}

\end{lem}
The proof of this Lemma is similar to that of \cite[Theorem 1.2]{BCS97} so we omit it here.

Inserting $\rho^{\infty}(t,q,p)=g_\infty(t,q)Z_\lambda^{-1}\exp\left\{-\frac{p^2}{2\lambda}\right\}$ in the formula in Lemma \ref{tyler}, we find
\begin{equation}
\label{sandra}
\frac{\partial}{\partial t}g_\infty=-p\left[\nabla_q g_\infty+\frac{1}{\lambda}\left(\nabla_qV(q)+\nabla_qF\ast g_\infty(q)\right)g_\infty\right]\,.
\end{equation}
We take the derivative with respect to $p$ in \eqref{sandra} and we obtain
\begin{equation*}
\nabla_q g_\infty+\frac{1}{\lambda}\left(\nabla_qV(q)+\nabla_qF\ast g_\infty(q)\right)g_\infty=0\,,
\end{equation*}
so that $g_\infty(t,q)=C(t)\exp\left\{-\frac{1}{\lambda}\left(V(q)+F\ast\rho^\infty(q)\right)\right\}$. Inserting this in \eqref{sandra}, we deduce $\frac{\partial}{\partial t}g_\infty=0$. Consequently, we have
\begin{equation}
\label{stannis}
\rho^\infty(t,q,p)=C\exp\left\{-\frac{1}{\lambda}\left(\frac{p^2}{2}+V(q)+F\ast\rho^\infty(q)\right)\right\}\,.
\end{equation}
This implies that $\rho^\infty\left(t,.,.\right)$ is an invariant probability for any $t\in[0;T]$. In particular, $\rho^\infty(0,.,.)$ is an invariant probability.

By making a compilation of all the previous results, we obtain

\begin{prop}
\label{manue}
For any sequence of positive reals $(t_n)_n$ which converges to $+\infty$, we can extract a subsequence which converges to an invariant probability.
\end{prop}

We are now at the position to prove Theorem \ref{theo: main result}.
\begin{proof}[Proof of Theorem \ref{theo: main result}]
If there was a unique invariant probability, the convergence would be proven. However, in our setting, it is possible that there are several invariant probabilities, see \cite{DuongTugaut2016}.

Let us assume that there are more than one adherence values. These adherence values are necessarily invariant probabilities thanks to Proposition \ref{manue}. We deduce the first point of Theorem A, that is $\mathcal{A}_\sigma\subset\mathcal{S}_\sigma$. We now show that $\mathcal{A}_\sigma$ is path-connected. We proceed like in \cite{AOP,SPA}.

Set $\varphi_0$ a test function ($\mathcal{C}^\infty$ and with compact support) such that

\begin{equation*}
\int_{\bRb^d\times\bRb^d}\varphi_0(q,p)\rho^\infty(q,p)dqdp\neq0
\end{equation*}
for any $\rho^\infty\in\mathcal{A}_\sigma$ and such that there exist $\rho_1^\infty,\rho_2^\infty\in\mathcal{A}_\sigma$ satisfying

\begin{equation*}
\int_{\bRb^d\times\bRb^d}\varphi_0(q,p)\rho_1^\infty(q,p)dqdp<0<\int_{\bRb^d\times\bRb^d}\varphi_0(q,p)\rho_2^\infty(q,p)dqdp\,.
\end{equation*}

Since $\rho_1^\infty$ and $\rho_2^\infty$ are two adherence values, there exists a sequence $(t_n)$ which converges towards $\infty$ such that 
\begin{equation*}
\int_{\bRb^d\times\bRb^d}\varphi_0(q,p)\rho_{t_n}(q,p)dqdp=0\,.
\end{equation*}
According to Proposition \ref{manue}, there exists a subsequence $\rho_{t_{\varphi(n)}}$ which converges weakly towards an element of $\mathcal{A}_\sigma\subset\mathcal{S}_\sigma$. However, the integral of $\varphi_0$ with respect to this limiting value is $0$, which is absurd.

The proof of Theorem \ref{theo: main result} is achieved.
\end{proof}

\textbf{Acknowledgements}. M. H. Duong was supported by ERC Starting Grant 335120.
\bibliographystyle{plain}

\end{document}